\documentclass[12pt]{amsart}
\usepackage{amsmath,amsthm,amssymb,amsfonts}
\usepackage[nobysame,abbrev,alphabetic]{amsrefs}

\numberwithin{equation}{section}

\DeclareFontFamily{OT1}{rsfs}{}
\DeclareFontShape{OT1}{rsfs}{n}{it}{<-> rsfs10}{}
\DeclareMathAlphabet{\mathscr}{OT1}{rsfs}{n}{it}

\theoremstyle{definition}

\newcommand{\IGNORE}[1]{}

\newtheorem{theorem}{Theorem}[section]
\newtheorem{proposition}[theorem]{Proposition}

\newtheorem{lemma}[theorem]{Lemma}

\theoremstyle{definition}
\newtheorem{definition}[theorem]{Definition}
\newtheorem{remark}[theorem]{Remark}

\def\mR {\mathbb {R}^n}

\newcommand{\dstyle}{\displaystyle}
\newcommand{\hide}[1]{}

\begin{document}

\title[Finite total $Q$-curvature]{Fractional Poincar\'e inequality with finite total $Q$-curvature}

\author[Yannick Sire]{Yannick Sire}\address{Yannick Sire, Department of Mathematics, Johns Hopkins University, 404 Krieger Hall, 3400 N. Charles Street, Baltimore 21218, USA} \address{ email: sire@math.jhu.edu}

\author[Yi Wang]{Yi Wang}
\address{Yi Wang, Department of Mathematics, Johns Hopkins University, 404 Krieger Hall, 3400 N. Charles Street, Baltimore 21218, USA}
\address{ email: ywang@math.jhu.edu}
\setcounter{page}{1}
\thanks{The research of the second author is partially supported
by NSF grant DMS-1547878}


\begin{abstract}In this paper, we prove several Poincar\'e inequalities of fractional type on conformally flat manifolds with finite total $Q$-curvature. This shows a new aspect of the $Q$-curvature on noncompact complete manifolds.

\end{abstract}
\maketitle

\section{Introduction}

The $Q$-curvature arises naturally as a conformal invariant associated to
the Paneitz operator. When $n=4$, the Paneitz operator is defined as:
$$P_g=\Delta^2+\delta(\frac{2}{3}Rg-2 Ric)d,$$
where $\delta$ is the divergence, $d$ is the differential, $R$ is the scalar curvature of $g$, and $Ric$
is the Ricci curvature tensor. The Paneitz $Q$-curvature is defined as
$$Q_g=\frac{1}{12}\left\{-\Delta R +\frac{1}{4}R^2 -3|E|^2 ,\right\}
$$
where $E$ is the traceless part of $Ric$, and $|\cdot|$ is taken with respect to the metric $g$.
Under the conformal change $g_{u}=e^{2u}g_0$, the Paneitz operator transforms by $P_{g_u}=e^{-4u}P_{g_0}$,
and
$Q_{g_u}$ satisfies the fourth order equation
\begin{equation}
\label{1.88}P_{g_{0}}u+2Q_{g_0}=2Q_{g_{u}}e^{4u}.\end{equation}
This is analogous to the transformation law satisfied by the Laplacian operator $-\Delta_g$ and the Gaussian curvature $K_g$ on surfaces,
$$-\Delta_{g_0}u+K_{g_0}=K_{g_u}e^{2u}.$$

The invariance of $Q$-curvature in dimension $4$ is due to the Chern-Gauss-Bonnet formula for a closed manifold $M$:
\begin{equation}\label{GBCEq}\chi(M)=\dstyle\frac{1}{4\pi^2} \int_{M}\left(\frac{|W|^2}{8}+Q_g\right) dv_g,\end{equation}
where $W$ denotes the Weyl tensor. 

A related problem is the classical isoperimetric inequality on a complete simply connected surface $M^2$, called Fiala-Huber's \cite{Fiala}, \cite{Huber} isoperimetric inequality
\begin{equation}\label{FialaHuber}
vol(\Omega)\leq \frac{1}{2(2\pi-\int_{M^2}K_g^+ dv_g)} Area(\partial \Omega)^2,
\end{equation}
where $K_g^+$ is the positive part of the Gaussian curvature $K_g$. Also $\int_{M^2}K_g^+ dv_g< 2\pi$ is the sharp bound for the isoperimetric inequality to hold. 

In \cite{YW15}, we generalize the Fiala-Huber's isoperimetric inequality to all even dimensions, replacing the role of the Gaussian curvature in dimension two by that of the $Q$-curvature in higher dimensions:

Let $(M^n,g)=(\mR, g= e^{2u}|dx|^2)$ be a complete noncompact even dimensional manifold.
Let $Q^+$ and $Q^-$ denote the
positive and negative part of $Q_g$ respectively; and $dv_g$ denote the volume form of $M$. Suppose $g= e^{2u}|dx|^2$ is a ``normal" metric, i.e.
\begin{equation}\label{normal}u(x)=
\displaystyle \frac{1}{c_n}\int_{\mR} \log \frac{|y|}{|x-y|} Q_{g}(y) dv_g(y) + C;
\end{equation}
for some constant $C$.
If
\begin{equation}\label{assumption1}
\beta^+:= \int_{M^n}Q^{+}dv_g < c_n
\end{equation}
where $c_n=2^{n-2}(\frac{n-2}{2})!\pi^{\frac{n}{2}}$,
and \begin{equation}\label{assumption2}
\beta^-:=\int_{M^n}Q^{-}dv_g < \infty,
\end{equation}
then $(M^n,g)$ satisfies the isoperimetric inequality with isoperimetric constant depending only on $n, \beta^+$ and
$\beta^-$.
Namely, for any bounded domain $\Omega\subset M^n$ with smooth boundary,
\begin{equation}\label{1.89}
|\Omega|_g^{\frac{n-1}{n}}\leq C(n, \beta^+,\beta^-) |\partial \Omega |_g.
\end{equation}

The main purpose of the current paper is to further derive the fractional Poincar\'{e}  inequality. The constant in this inequality is also controlled by the integral of the $Q$-curvature.

\begin{theorem}\label{Theorem1}Let $(M^n,g)=(\mR, g= e^{2u}|dx|^2)$ be a complete noncompact even dimensional manifold.
Let $Q^+$ and $Q^-$ denote the
positive and negative part of $Q_g$ respectively; and $dv_g$ denote the volume form of $M$. Suppose $g= e^{2u}|dx|^2$ is a ``normal" metric, i.e.
\begin{equation}\label{normal}u(x)=
\displaystyle \frac{1}{c_n}\int_{\mR} \log \frac{|y|}{|x-y|} Q_{g}(y) dv_g(y) + C;
\end{equation}
for some constant $C$.
If
\begin{equation}\label{assumption1}
\beta^+:= \int_{M^n}Q^{+}dv_g < c_n
\end{equation}
where $c_n=2^{n-2}(\frac{n-2}{2})!\pi^{\frac{n}{2}}$,
and \begin{equation}\label{assumption2}
\beta^-:=\int_{M^n}Q^{-}dv_g < \infty,
\end{equation}
then $(M^n,g)$ satisfies the fractional Poincar\'{e} inequality with constant depending only on $n, \beta^+$ and
$\beta^-$.
Namely, for $\alpha \in (0,2)$, there exists $C>0$ depending only on $n, \beta^+$ and
$\beta^-$, such that for any function $f$ in $C^2(M^n)$ and any Euclidean ball $B$,
\begin{equation}
\int_B |f(x)-f_{B, \omega}|^2\omega(x)dx
\leq C\int_{2B}\int_{2B}
\frac{|f(x)-f(y)|^2}{d_g(x,y)^{n+\alpha}}
\omega(x)\omega(y)dxdy.
\end{equation}
\end{theorem}

The proof of Theorem \ref{Theorem1} is based on several steps: the first one consists in deriving a $2$-Poincar\'e inequality relying on the properties of the metric under consideration. This step gives in particular an important geometric meaning of the $Q$-curvature. The second step consists in using functional calculus to rewrite properly this Poincar\'e inequality and using spectral theory to estimate powers of a suitable weighted Laplacian. Finally, one derives the desired inequality in Theorem \ref{Theorem1} by a covering argument and some estimates. This approach has been successfully used in \cite{MRS,RS} to derive some types of fractional Poincar\'e inequalities in some Euclidean or geometric contexts. Notice that in these latter works, the fractional Poincar\'e inequality is not symmetric with respect to the measure in the right hand side. In our Theorem, this is the case due to a suitable covering and also the fact that we are considering local estimates. 

\begin{remark}We remark that $\liminf_{|x|\rightarrow \infty}R_g(x)\geq 0$ would imply $g= e^{2u}|dx|^2$ is a ``normal" metric in dimension four. See \cite{CQY1}.We also remark that the constant $c_n$ in the assumption (\ref{assumption1}) is sharp. In fact, $c_n$ is equal to the integral of the $Q$-curvature
on a half cylinder (a cylinder with a round cap attached to one of its two ends); but obviously a half cylinder fails to satisfy the isoperimetric inequality. We also remark that being a normal metric is a natural and necessary assumption. 
\end{remark}

\begin{remark}
It is worth noting that normal metric is a necessary assumption in this theorem, because without it, there exist quadratic functions in the kernel of the bi-Laplacian operator $\Delta^2$ (with respect to the flat metric) for which (\ref{1.3}) fails. 
\end{remark}
The paper will be organized as follows. In section 2, we present preliminaries on the $Q$-curvature in conformal geometry. In section 3, we discuss $A_p$ weights, their properties and relations to various inequalities. We then devote section 4 to the volume growth estimate of geodesic balls, which will be used in later sections, and discuss the $p$-Poincar\'{e} inequality in section 5. Finally, in section 6, we finish the proof of Theorem \ref{Theorem1}.\\

\textbf{Acknowledgments:}
The second author is grateful to Alice Chang and Paul Yang for discussions and interest to this work.



\section{preliminaries in conformal geometry}
In the past decades, there have been many works focusing on the study of the $Q$-curvature equation and the associated
conformal covariant operators, both from PDE point of view and from the geometry point of view.
We now discuss some background of it in conformal geometry.
Consider a 4-manifold $(M^4,g)$, the Branson's $Q$-curvature of $g$
is defined as
$$Q_g:=\frac{1}{12}\left\{-\Delta R_g +\frac{1}{4}R_g^2 -3|E|^2 ,\right\}
$$
where $R_g$ is the scalar curvature, $E_g$ is the traceless part of $Ric_g$, and $|\cdot|$ is taken with respect to the metric $g$.
It is well known that the $Q$-curvature is an integral conformal invariant associated to
the fourth order Paneitz operator $P_g$
$$P_g:=\Delta^2+\delta(\frac{2}{3}R_g g-2 Ric_g)d.$$
Under the conformal change 
$Q_{g_u}$ satisfies the fourth order differential equation.
In the particular situation when the background metric $g_0=|dx|^2$, the equation (\ref{1.88}) reduces to
$$(-\Delta)^2 u= 2Q_{g_{u}}e^{4u},$$ where $\Delta$ is the Laplacian operator of the flat metric $g_0$.

Another analogy between the $Q$-curvature and the Gaussian curvature is the invariance of the integral of the $Q$-curvature, due to the Chern-Gauss-Bonnet formula for closed manifold $M^4$:
\begin{equation}\label{GB}\chi(M^4)=\dstyle\frac{1}{4\pi^2} \int_{M^4}\left(\frac{|W_g|^2}{8}+Q_g\right) dv_{M},\end{equation}
where $W_g$ denotes the Weyl tensor.
There has been great progress in understanding the $Q$-curvature. For example see the work of Fefferman-Graham \cite{FeffermanGraham} on the study of the $Q$-curvature and ambient metrics, that of Chang-Qing-Yang \cite{CQY1} on the $Q$-curvature and Cohn-Vossen inequality; and that of Malchiodi \cite{Malchiodi}, Chang-Gursky-Yang \cite{ChangGurskyYang}
on the existence and regularity of constant $Q$-curvature metrics, etc. 

For higher dimensions, the $Q$-curvature is defined via the analytic continuation in the dimension and the formula is not explicit in general.
However when the background metric is flat, it satisfies, under the conformal change of metric $g_{u}=e^{2u}|dx|^2$, the $n$-th order differential equation
$$(-\Delta)^{\frac{n}{2}} u= 2Q_{g_{u}}e^{nu},$$ where $\Delta$ is the Laplacian operator of $|dx|^2$.

For complete conformally flat manifolds, the Gauss-Bonnet formula \eqref{GB} is no longer valid.
Chang-Qing-Yang \cite{CQY1} proved instead the following result: let $(M^4,g)=(\mathbb{R}^4, e^{2u} |dx|^2)$ be a noncompact complete conformally flat manifold
with finite total $Q$-curvature, i.e. $\int_{M^4}|Q_g|d v_g<\infty$. If the metric is normal, i.e.
\begin{equation}\label{wFlat}
u(x)=\frac{1}{4\pi^2}\int_{\mathbb{R}^4}\log\frac{|y|}{|x-y|}Q_g(y)e^{4u(y)} dy + C,
\end{equation} or if the scalar curvature $R_g$ is nonnegative at infinity, then
\begin{equation}\label{1.3}
\dstyle \frac{1}{4\pi^2}\int_{M^4} Q_g dv_g\leq  \chi(\mathbb{R}^4)=1,
\end{equation}
and \begin{equation}\label{1.4}
\displaystyle \chi(\mathbb{R}^4)-\frac{1}{4\pi^2}\int_{\mathbb{R}^4} Q_g dv_g= \sum_{j=1}^{k}\lim_{r\rightarrow
\infty }\frac{vol_g(\partial B_{j}(r))^{4/3}}{4(2\pi^2)^{1/3}vol_g(B_{j}(r))},
\end{equation}
\noindent where $B_{j}(r)$ denotes the Euclidean ball with radius $r$ at the $j$-th end.

Chang, Qing and Yang's theorem asserts that for $4$-manifolds (in fact, their theorem is valid for all even dimensions) which is conformal to the Euclidean space, the integral of the $Q$-curvature controls the asymptotic isoperimetric ratio at the end of this complete manifold.
This is analogous to the two-dimensional result by
Cohn-Vossen \cite{Cohn-Vossen}, who studied the Gauss-Bonnet integral for a noncompact complete surface $M^2$ with analytic
metric. He showed that if the Gaussian curvature $K_g$ is absolutely
integrable (in which case we say the manifold has finite total curvature), then
\begin{equation}\label{1.1}
\dstyle \frac{1}{2\pi}\int_{M} K_g dv_g\leq  \chi(M),
\end{equation}
where $\chi(M)$ is the Euler characteristic of $M$. Later, Huber \cite{Huber} extended this inequality to
metrics with much weaker regularity. More importantly, he proved that such a surface
$M^2$ is conformally equivalent to a closed surface with finitely many points removed. The difference of the two sides in inequality (\ref{1.1}) encodes the asymptotic behavior of the manifold at its ends.  The precise geometric interpretation has been given by Finn \cite{Finn} as follows. Suppose a noncompact complete surface has absolutely integrable Gaussian curvature. Then one may
represent each end conformally as $\mathbb{R}^2 \setminus K$ for some compact set $K$. Define the asymptotic
isoperimetric constant of the $j$-th end to be
$$\nu_i=\lim_{r\rightarrow \infty}\frac{L_g^2(\partial B(0,r)\setminus K)}{4\pi A_g(B(0,r)\setminus K)},
$$
where $B(0,r)$ is the Euclidean ball centered at origin with radius $r$, $L$ is the
length of the boundary, and $A$ is the area of the domain. Then
\begin{equation}\label{1.2}
\displaystyle \chi(M)-\frac{1}{2\pi}\int_{M} K_g dv_g= \sum_{j=1}^{N}\nu_{j},
\end{equation}
where $N$ is the number of ends on $M$. This result tells us that the condition of finite total
Gaussian curvature has rigid geometric and analytical consequences.

Chang, Qing and Yang's results
(\ref {1.3}), (\ref {1.4}) are higher dimensional counterparts of (\ref {1.1}), (\ref {1.2}). 

\section{$A_p$ weights and Strong $A_\infty$ weights}
In this section, we are going to present the definitions and the properties of $A_p$ weights and strong $A_\infty$ weights.

In harmonic analysis, $A_p$ weights ($p\geq 1$) are introduced to characterize when a function $\omega$ could be a weight such that the associated measure $\omega(x)dx$
has the property that the maximal function $\textrm{M}$ of an $L^1$ function is weakly $L^1$, and that the maximal
function of an $L^p$ function is $L^p$ if $p>1$.

For a nonnegative locally integrable function $\omega$, we call it an $A_p$ weight $p>1$, if
\begin{equation}\label{Ap}
\frac{1}{|B|}\int_{B}\omega(x)dx\cdot \left(\frac{1}{|B|}\int_{B}\omega(x)^{-{p'}/{p}}dx\right)^{{p}/{p'}}\leq C<\infty,
\end{equation}
for all balls $B$ in $\mathbb{R}^n$. Here $p'$ is conjugate to $p$: $\frac{1}{p'}+\frac{1}{p}=1$. The constant $C$ is uniform for all $B$ and we call the smallest such constant $C$ the $A_p$ bound of $\omega$.
The definition of $A_1$ weight is given by taking limit of $p\rightarrow 1$ in (\ref{Ap}),
which gives
$$\dstyle \frac{1}{|B|} \int_{B}\omega  \leq C \omega(x), $$
for almost all $x\in B$.
Thus it is equivalent to say the maximal function of the weight is bounded by the weight itself:
$$\textsl{M}\omega(x)\leq C' \omega(x),$$
for a uniform constant $C'$.
Another extreme case is the $A_\infty$ weight. $\omega$ is called an $A_\infty$ weight if it is an $A_p$ weight for some $p>1$.
It is not difficult to see $A_1\subseteq A_p\subseteq A_{p'}\subseteq A_\infty$ when $1\leq p\leq p'\leq \infty$.

One of the most fundamental property of $A_p$ weight is the reverse H\"{o}lder inequality:
if $\omega$ is $A_p$ weight for some $p\geq 1$, then there exists an $r>1$ and a $C>0$, such that
\begin{equation}\label{reverse holder}
\dstyle \left( \dstyle \frac{1}{|B|} \int_{B}\omega^r dx \right)^{1/r}\leq \frac{C}{|B|} \dstyle \int_{B}\omega dx,
\end{equation}
for all balls $B$.
This would imply that any $A_p$ weight $\omega$ satisfies the doubling property:
there is a $C>0$ (it might be different from the constant $C$ in (\ref{reverse holder})), such that $$\dstyle \int_{B(x_0, 2r)} \omega(x) dx\leq C \int_{B(x_0, r)} \omega(x) dx$$
for all balls $B(x_0, r)\subset\mR$.

Suppose $\omega_1$ and $\omega_2$ are $A_1$ weights, and let $t$ be any positive real number. Then it is not hard to show that $\omega_1\omega_2^{-t}$ is an $A_\infty$ weight.
Conversely, the factorization theorem of $A_\infty$ weight proved by Peter Jones \cite{Jones} asserts: if $\omega$ is an $A_\infty$ weight, then there exist $\omega_1$ and $\omega_2$ which are both $A_1$ weights, and $t>1$ such that $\omega=\omega_1\omega_2^{-t}.$ Later, in the proof of the main theorem, we will decompose the volume form $e^{nu}$ into two pieces. The idea to decompose $e^{nu}$ is inspired by Peter Jones' factorization theorem. In our case, we give an explicit decomposition of the weight $e^{nu}$, and by analyzing each part in the decomposition we finally prove that $e^{nu}$ is a strong $A_\infty$ weight, a class of weights much stronger than $A_\infty$ that we will introduce in the following.

The notion of strong $A_\infty$ weight was first proposed by David and Semmes in \cite{DS1}.
Given a positive continuous weight $\omega$, we define $\delta_\omega(x,y)$ to be:
\begin{equation}\label{def1}
\delta_\omega(x,y):=\left(\int_{B_{xy}}\omega(z)dz\right)^{1/n},
\end{equation}
where $B_{xy}$ is the ball with diameter $|x-y|$ that contains $x$ and $y$.
One can prove that $\delta_\omega$ is only a quasi-distance in the sense that it satisfies the quasi-triangle inequality
$$\delta_\omega(x,y)\leq C(\delta_\omega(x,z)+ \delta_\omega(z,y)).$$
\noindent On the other hand, for a continuous function $\omega$, by taking infimum over all rectifiable arc $\gamma\subset B_{xy}$ connecting $x$ and $y$, one can define the $\omega$-distance to be
\begin{equation}\label{def2}d_\omega(x,y):=\inf_{\gamma}\int_\gamma \omega^{\frac{1}{n}}(s)|ds|.\end{equation}

If $\omega$ is an $A_\infty$ weight, then it is easy to prove (see for example Proposition 3.12 in \cite{S2})
\begin{equation}\label{A}
d_\omega(x,y)\leq C\delta_\omega(x,y)
\end{equation}
for all $x,y\in \mathbb{R}^n$.
If in addition to the above inequality, $\omega$ also satisfies the reverse inequality, i.e.
\begin{equation}\label{strong A}
\delta_\omega(x,y)\leq Cd_\omega(x,y),
\end{equation}
for all $x,y\in \mathbb{R}^n$, then we say $\omega$ is a strong $A_\infty$ weight, and $C$ is the bound of this strong $A_\infty$ weight.


Every $A_1$ weight is a strong $A_\infty$ weight, but for any $p>1$ there is an $A_p$ weight which is not strong $A_\infty$. Conversely, for
any $p>1$ there is a strong $A_\infty$ weight which is not $A_p$. It is easy to verify by definition the function $|x|^{\alpha}$ is $A_1$ thus strong $A_\infty$ if $-n<\alpha\leq 0$; it is not $A_1$ but still strong $A_\infty$ if $\alpha>0$.
And $|x_1|^\alpha$ is not strong $A_\infty$ for any $\alpha>0$ as one can choose a curve $\gamma$ contained in the
$x_2$-axis.

The notion of strong $A_\infty$ weight was initially introduced in order to study weights that are comparable to the Jacobian of quasi-conformal maps.
It was proved by Gehring that the Jacobian of a quasiconformal map on $\mR$ is always a strong $A_\infty$ weight, and it was conjectured that the converse was assertive: every strong $A_\infty$ weight is comparable to the Jacobian
of a quasi-conformal map. Later, however, counter-examples were found by Semmes \cite{S3} in dimension $n\geq 3$, and by Laakso \cite{La} in dimension 2.
Nevertheless, it was proved by David and Semme that a strong $A_\infty$ weight satisfies the Sobolev inequality:
\begin{theorem}\label{2.1}\cite{DS1} Let $\omega$ be a strong $A_\infty$ weight. Then for $f\in C^{\infty}_0(\mR)$,
\begin{equation}\dstyle \left( \int_{\mathbb{R}^n}|f(x)|^{p^*}\omega(x)dx\right)^{1/{p^*}}\leq C \dstyle \left(\int_{\mathbb{R}^n}
(\omega^{-\frac{1}{n}}(x)|\nabla f(x)|)^p\omega(x)dx \right)^{1/p},
\end{equation}
where $1\leq p<n$, $p^*=\frac{np}{n-p}$. Take $p=1$, it is the standard isoperimetric inequality. The constant $C$ in the inequality only depends on the strong $A_\infty$ bound of $\omega$ and $n$.
\end{theorem}


By taking $f$ to be a smooth approximation of the indicator function of domain $\Omega$, this implies the validity of the isoperimetric inequality with respect to the weight $\omega$.
In this paper, we will take $\omega= e^{nu}$, the volume form of $(\mR,e^{2u}|dx|^2)$. We aim to show $e^{nu}$ is a strong $A_\infty$ weight. By Theorem \ref{2.1}, this implies the isoperimetric inequality on $(\mR,e^{2u}|dx|^2)$:
$$ \dstyle (\int_{\Omega} e^{nu(x)}dx)^{\frac{n-1}{n}}\leq C \int_{\partial \Omega} e^{(n-1)u(x)}d\sigma_x,$$
or equivalently, for $g=e^{2u}|dx|^2$,
$$|\Omega|_g^{\frac{n-1}{n}}\leq C |\partial \Omega|_g.$$

A good reference for $A_p$ weights is Chapter 5 in \cite{Stein}.
For more details on strong $A_\infty$ weight, we refer the readers to \cite{DS1}, where the concept was initially proposed.

\section{Volume growth of geodesic balls}
We will study in this section that the volume growth of geodesic balls is Euclidean.
There are two different cases.
In Case 1, we suppose $\dstyle \int_M Q^+_g dv_g< 4\pi^2$, and $\dstyle \int_M Q^-_g dv_g<\infty$. Then the volume growth is Euclidean
\begin{equation}
C_2 r^n \leq Vol_g (B^g(0,r))\leq C_1 r^n,
\end{equation}
and the constants are uniformly controlled by $n$ and the integral of $Q$-curvature. More precisely, $C_i$, $i=1,2$ depend only on $n$, $\dstyle 4\pi^2-\int_M Q^+_g dv_g>0$, $\dstyle \int_M Q^-_g dv_g$. 
This uniform result is derived from strong $A_\infty$ property of the conformal factor $e^{nu}$, which was proved in \cite{YW15}.
In the other case, which we call Case 2, we assume weaker assumptions on the integrals of the $Q$-curvature, namely $\dstyle \int_M Q_g dv_g< 4\pi^2$, and $\dstyle \int_M |Q_g| dv_g<\infty$.
As a consequence, we show that the volume growth is still Euclidean. However, the constants are not uniformly controlled by the integrals of the $Q$-curvature.
In our proof of Theorem \ref{Theorem1}, since we need to localize arguments on each end of the locally conformally flat manifold, we will apply
the consequence of Case 2. Case 1 is a stronger result, and is of independent interest. Therefore, we also provide proof in
this section.

This result is going to be used in the proof of Theorem \ref{Theorem1}.
\begin{proposition}\label{Proposition1}
Let $B^g(x_0,r)$ be a geodesic ball centered at $x_0$, with radius $r$ measured by the metric $g$. Then
$$C_2 r^n \leq Vol_g (B^g(x_0,r))\leq C_1 r^n,$$
where $C_i$, $i=1,2$ depend only on $n$ and $\dstyle \int_M |Q_g| dv_g$.
\end{proposition}
\begin{proof}
By the main theorem of \cite{YW15}, on a conformally flat manifold $M$ with totally finite $Q$-curvature and normal metric $g=e^{2u} |dx|^2$,
the isoperimetric inequality is valid. Moreover, the isoperimetric constant is uniformly controlled by $n$ and $\dstyle \int_M |Q_g| dv_g$.
This gives directly the lower bound of volume growth of geodesic balls. Namely, there exists $C_2$, depending only on the isoperimetric constant,
thus only on $n$ and $\dstyle \int_M |Q_g| dv_g$, such that $$C_2 r^n \leq Vol_g (B^g(0,r)).$$
On the other hand, \cite{YW15} also proves that the volume form $e^{nu}$ is a strong $A_\infty$ weight.
We recall the definition of strong $A_\infty$ weights.
\begin{definition}
Given a positive continuous weight $\omega$, we define $\delta_\omega(x,y)$ to be:
\begin{equation}\label{def1}
\delta_\omega(x,y):=\left(\int_{B_{xy}}\omega(z)dz\right)^{1/n},
\end{equation}
where $B_{xy}$ is the ball with diameter $|x-y|$ that contains $x$ and $y$.
\noindent On the other hand, for a continuous function $\omega$, by taking infimum over all rectifiable arc $\gamma\subset B_{xy}$ connecting $x$ and $y$, one can define the $\omega$-distance to be
\begin{equation}\label{def2}d_\omega(x,y):=\inf_{\gamma}\int_\gamma \omega^{\frac{1}{n}}(s)|ds|.\end{equation}
If $\omega$ satisfies
\begin{equation}\label{A}
d_\omega(x,y)\leq C\delta_\omega(x,y);
\end{equation}
and the reverse inequality
\begin{equation}\label{strong A}
\delta_\omega(x,y)\leq Cd_\omega(x,y),
\end{equation}
for all $x,y\in \mathbb{R}^n$, then we say $\omega$ is a strong $A_\infty$ weight, and $C$ is the bound of this strong $A_\infty$ weight.
\end{definition}

Being a strong $A_\infty$ weight, $e^{nu}$ relates the volume of a geodesic sphere and the distance function.
Given a geodesic ball $B^g(0,r)$, the Euclidean diameter of this ball is realized by two points $x$, $y$ on the boundary $\partial B^g(0,r)$.
Thus $B^g(x_0,r)\subset B^0(x, R)$, where $R$ is the Euclidean distance of $x$, $y$, and $B^0(x,R)$
denotes the Euclidean ball centered at $x$ with radius $R$.

Let us denote by $p$ the middle point of $x$ and $y$. Then by the doubling property of strong $A_\infty$ weight $e^{nu}$,
\begin{equation}
Vol_g(B^g(0,r)\leq  Vol_g(B^0(x, R))  \leq  C_3 Vol_g (B^0(p, R/2)).
\end{equation}
Notice $x$, $y$ lie on $\partial B^0(p, R/2)$, and the line segment between them is the diameter of this Euclidean ball.
By the definition of strong $A_\infty$ weight,
\begin{equation}
Vol_g (B^0(p, R/2))\leq C_4 d_g(x,y)^n.
\end{equation}
Using the triangle inequality, it is obvious that
\begin{equation}d_g(x,y)\leq 2 r. \end{equation}
Thus we obtain
\begin{equation}
Vol_g(B^g(x_0,r)\leq C_1 r^n.\end{equation}
$C_3$ depends only on the strong $A_\infty$ bound of $e^{nu}$, and $C_4$ depends on $n$.
Thus $C_1$ that is determined by $C_3$ and $C_4$ depends only on $n$ and $\dstyle \int_M |Q_g| dv_g$.
\end{proof}

\section{$p$-Poincar\'{e} inequality}

By \cite{YW15}, $\omega=e^{nu}$ is strong $A_\infty$. Thus by \cite{DS1},
we have pointwise Poincar\'{e} inequality: for any $x$, $y\in B$, 
\begin{equation}\label{pointwisePoincare}
|f(x)-f(y)|\leq C \int_{B_{xy}} (\omega(B_{xu})^{-\frac{n-1}{n}}+\omega(B_{yu})^{-\frac{n-1}{n}})
|\nabla f(u)| \omega(u)^{\frac{n-1}{n}}du.
\end{equation}
Here we denote by $B_{xy}$ the smallest Euclidean ball that coveres $x$ and $y$.

\begin{lemma}
For a strong $A_\infty$ weight $\omega$, the above pointwise Poincar\'{e} inequality
implies the $p$-Poincar\'{e} inequality, $p>1$: For any Euclidean ball $B$, let $2B$ denote the concentric ball with double radius. Then
\begin{equation}\label{pPoincare}
\begin{split}
&\int_B|f(x)-f_{B,\omega}|^p\omega(x)dx\\
\leq &C\omega(B)^{\frac{p}{n}} \int_{2B} 
|\nabla f(u)|^p \omega(u)^{1-\frac{p}{n}}du,
\end{split}
\end{equation}
where $\displaystyle f_{B,\omega}=\frac{1}{\omega(B)}\int_B f(x)\omega(x)dx$.
\end{lemma}

\begin{proof}
 By pointwise Poincar\'{e} inequality \eqref{pointwisePoincare}
 \begin{equation}
|f(x)-f(y)|^p\leq C (\int_{B_{xy}} (\omega(B_{xu})^{-\frac{n-1}{n}}+\omega(B_{yu})^{-\frac{n-1}{n}})
|\nabla f(u)|  \omega(u)^{\frac{n-1}{n}}du)^p.
\end{equation}
Applying H\"{o}lder's inequality, we obtain
 \begin{equation}\label{5.1}
\begin{split}
|f(x)-f(y)|^p\leq & C \Big(\int_{B_{xy}} 
|\nabla f(u)|^p \omega(u)^{\frac{-p}{n}}   \Big(\omega(B_{xu})^{-\frac{n-1}{n}}+\omega(B_{yu})^{-\frac{n-1}{n}}\Big) \omega(u) du\Big)\\
& \cdot \Big(\int_{B_{xy}}\Big(\omega(B_{xu})^{-\frac{n-1}{n}}+\omega(B_{yu})^{-\frac{n-1}{n}}\Big) \omega(u)du\Big)^{p-1}.
\end{split}
\end{equation}

\begin{equation}
\begin{split}
\int_{2B}\omega{(B_{xu})}^{-\frac{n-1}{n}}  \omega(u)du
\leq & C \sum_{k\geq 0}\int_{u\in 2B, 
\omega_{(B_{xu})}\approx 2^{-k} \omega(2B)}
2^{\frac{k(n-1)}{n}}\omega(2B)^{-\frac{n-1}{n}}\omega(u)du\\
\leq& C \omega(B)^{-\frac{n-1}{n}}
\sum_{k\geq 0} 2^{\frac{k(n-1)}{n}}2^{-k}\omega(B)\\
=&C \omega(B)^{\frac{1}{n}}.\\
\end{split}
\end{equation}
Therefore,
\begin{equation}\label{5.3}
\begin{split}
&\int_{B_{xy}}\omega(B_{xu})^{-\frac{n-1}{n}}+\omega(B_{yu})^{-\frac{n-1}{n}}\Big) \omega(u)du\\
\leq & 2C\int_{2B}\omega(B_{xu})^{-\frac{n-1}{n}} \omega(u)du\\
\leq & C \omega(B)^{\frac{1}{n}}.\\
\end{split}
\end{equation}
Using it in \eqref{5.1}, we have 
\begin{equation}\label{5.1}
\begin{split}
& |f(x)-f(y)|^p\\
\leq & C\omega(B)^{\frac{p-1}{n}} \int_{B_{xy}} 
|\nabla f(u)|^p \omega(u)^{\frac{-p}{n}}   \Big(\omega(B_{xu})^{-\frac{n-1}{n}}+\omega(B_{yu})^{-\frac{n-1}{n}}\Big) \omega(u)du .\\
\end{split}
\end{equation}
Hence 
\begin{equation}\label{5.2}
\begin{split}
&\int_B\int_B|f(x)-f(y)|^p \omega(x)\omega(y)dxdy\\
\leq & C\omega(B)^{\frac{p-1}{n}} 
\int_{B_{xy}} \int_B\int_B
|\nabla f(u)|^p \omega(u)^{\frac{-p}{n}} 
\Big(\omega(B_{xu})^{-\frac{n-1}{n}}+\omega(B_{yu})^{-\frac{n-1}{n}}\Big)\\
&\hspace{30mm}  \omega(x)\omega(y)dxdy  \omega(u) du.\\
\end{split}
\end{equation}
We use inequality \eqref{5.3} again to see that 
\begin{equation}
\int_B\int_B
\Big(\omega(B_{xu})^{-\frac{n-1}{n}}+\omega(B_{yu})^{-\frac{n-1}{n}}\Big)  \omega(x)\omega(y)dxdy \leq 
C\omega(B)^{1+\frac{1}{n}}.
 \end{equation}
Therefore \eqref{5.2} becomes
\begin{equation}\label{strongPPoincare}
\begin{split}
&\int_B\int_B|f(x)-f(y)|^p \omega(x)\omega(y)dxdy\\
\leq & C\omega(B)^{1+\frac{p}{n}} 
\int_{B_{xy}} 
|\nabla f(u)|^p \omega(u)^{\frac{-p}{n}} 
\omega(u)du .\\
\end{split}
\end{equation}
Therefore, 
\begin{equation}
\begin{split}
&\int_B|f(x)-f_{B,\omega}|^p\omega(x)dx\\
\leq &C\omega(B)^{\frac{p}{n}} \int_{2B} 
|\nabla f(u)|^p \omega(u)^{1-\frac{p}{n}}du,
\end{split}
\end{equation}
where $\displaystyle f_{B,\omega}=\frac{1}{\omega(B)}\int_B f(x)\omega(x)dx$.\\
Since $g=e^{2u}|dx|^2$, $dv_g=e^{nu}dx$, this is equivalent to 
\begin{equation}
\int_B|f(x)-f_{B,\omega}|^pdv_g\leq C vol_g(B)^{\frac{p}{n}} \int_{2B} 
|\nabla_g f(u)|^p dv_g.
\end{equation}
\end{proof}
In particular, let $p=2$, we have
\begin{equation}\label{2Poincare}
\int_B|f(x)-f_{B,\omega}|^2dv_g\leq C vol_g(B)^{\frac{2}{n}} \int_{2B} 
|\nabla_g f(u)|^2 dv_g.
\end{equation}

We define the operator $m$ by multiplication with $\chi_B$
$$m (f):=f  \chi_B;$$
the measure $\mu_2$ by $$\mu_2:=\omega(x)\chi_{2B} dx;$$
and the weighted divergence operator $ L_{\mu_2}$ by $$ L_{\mu_2} f :=\Delta_g f
=\sum_{i=1}^n\frac{1}{\omega(x)}\partial_i (\omega(x)^{1-\frac{2}{n}} \partial_i f).$$
Thus \eqref{2Poincare} is equivalent to the following statement:
for $f_{B,\omega}=0$
\begin{equation}\label{measure}
\int_{\mathbb{R}^n} m(f)^2d\mu_2\leq C vol_g(B)^{\frac{2}{n}} \int_{\mathbb{R}^n} 
|L_{\mu_2}^{1/2} f |^2 d\mu_2.
\end{equation}
Thus $$\|m(f)\|_{L^2(\mathbb{R}^n, \mu_2)}\leq Cvol_g(B)^{\frac{1}{n}}  \|L_{\mu_2}^{1/2} f\|_{L^2(\mathbb{R}^n, \mu_2)}. $$
By spectral theory, 
$m(f)\leq C L_{\mu_2}^{1/2} f$ in $L^2(\mathbb{R}^n, \mu_2)$ implies
$ m^\alpha(f)\leq C L_{\mu_2}^{\alpha/2} f $ in $L^2(\mathbb{R}^n, \mu_2)$. Thus for $f$ satisfying $f_{B,\omega}=0$,
$$ \int_{\mathbb{R}^n}  m^\alpha(f)\cdot f d\mu_2\leq C \int_{\mathbb{R}^n} L_{\mu_2}^{\alpha/2} f \cdot f d\mu_2. $$
In other words, for any $f$
$$ 
(\int_B|f(x)-f_{B,\omega}|^2\omega(x)dx)^{1/2}=(\int_{\mathbb{R}^n}  m^\alpha(f)\cdot f d\mu_2)^{1/2}\leq C\|L^{\alpha/4}_{\mu_2}f\|_{L^2(\mathbb{R}^n, \mu_2)}.$$

\section{Control of $\|L^{\alpha/4}_{\mu_2}f\|_{L^2} $}

We start with several lemmas. The following lemma provides off-diagonal estimates. This is simply on an energy inequality and we refer the reader to \cite{RS} for a proof. 

\begin{lemma} \label{off} There exists $C$ with the following property: for all closed disjoint subsets
$E,F\subset G$ with $\mbox{d}(E,F)=:d>0$, all function $f\in
L^2(G,d\mu_M)$ supported in $E$ and all $t>0$,
$$
\left\Vert (\mbox{I}+t \, L_{\mu_2})^{-1}f\right\Vert_{L^2(F,\mu_2)}+\left\Vert t \,
 L_{\mu_2}(\mbox{I}+t \, L_{\mu_2})^{-1}f\right\Vert_{L^2(F,\mu_2)}\leq $$
 $$ 8 \, e^{-C \,
 \frac{d}{\sqrt t}} \left\Vert f\right\Vert_{L^2(E,\mu_2)}.
$$
\end{lemma}

The next lemma is just spectral theory (see \cite{RS} for a proof). 
\begin{lemma}
Let $\alpha \in (0,2)$. Let $\mathcal{D}(L_{\mu_2})$ be the domain of functions on which $L_{\mu_2}$
is well defined. There exists $C>0$ such that 
for all $f\in \mathcal{D}(L_{\mu_2})$,
\begin{equation}
\|L^{\alpha/4}_{\mu_2}f  \|^2_{L^2(\mathbb{R}^n, \mu_2)}\leq C
\displaystyle \int_0^{\infty}t^{-1-\alpha/2} \|tL_{\mu_2}(I+tL_{\mu_2})^{-1}f\|^2_{L^2(\mathbb{R}^n,\mu_2)}dt.
\end{equation}
\end{lemma}

\begin{lemma}\label{OverlapCovering}Fix $t>0$.
 $B^g(x^t_j, \sqrt{t})$, $j=1,...$ is a Vitali's covering
 of $(M^n, g)=(\mathbb{R}^n, e^{2u}|dx|^2)$, such that they are mutually disjoint, and
$$M^n= \bigcup_{j\in \mathbb{N}} B^g(x^t_j,2 \sqrt{t}). $$ Then there exists $\tilde{C}$ such that 
 for all $\theta>1$ and $x\in M^n$, there are at most $\tilde{C}\theta^{2k}$ indices $j$
satisfying $d_g(x,y)\leq \theta \sqrt{t}$.
\end{lemma}
\begin{proof}We follow the same proof as that in \cite{Kanai}, as well as in \cite{RS}.
Let $x\in M^n$ and 
$$I (x):= \left \{j\in \mathbb{N}; d_g(x-x_j^t)\leq \theta\sqrt{t}.  \right \}$$
For all $j\in I(x)$
$$B^g(x^t_j, \sqrt{t}) \subset B^g(x, (1+\theta) \sqrt{t}), $$
and $$B^g(x, \sqrt{t}) \subset B^g(x^t_j, (1+\theta) \sqrt{t}). $$
Also, by Proposition \ref{Proposition1}, the measure $\omega(x)dx$ has doubling property. Therefore,  
\begin{equation}
\begin{split}
I(x)\omega(B^g(x, \sqrt{t}))\leq &\sum_{j\in I(x)} \omega (B^g(x^t_j, (1+\theta)\sqrt{t}))\\
\leq & C(1+\theta)^k \sum_{j\in I(x)}\omega(x^t_j, \sqrt{t})\\
\leq & C(1+\theta)^k \omega(x,(1+\theta) \sqrt{t})\\
\leq & C(1+\theta)^{2k} \omega(x, \sqrt{t}).\\
\end{split}\end{equation}
\end{proof}

\begin{lemma}\label{3.112}
There exists $C>0$ such that for all $t>0$ and all $j\in\mathbb{N}$ 
A.
\begin{equation}
\begin{split}
\|g_0^{j,t}\|^2_{L^2(C^{j,t}_k, \mu_2) } \leq 
\frac{C}{\omega(B^g(\sqrt{t}))} \int_{B^g(x^t_j, 4\sqrt{t})} \int_{B^g(x^t_j, 4\sqrt{t})} |f(x)-f(y) |^2 d\mu_2(x)d\mu_2(y).\\
\end{split}
\end{equation}

B. For all $k\geq 1$, 
\begin{equation}
\begin{split}
& \|g_k^{j,t}\|^2_{L^2(C^{j,t}_k, d\mu_2) } \\
\leq & \frac{C}{\omega(B^g(2^k\sqrt{t}))} \int_{B^g(x^t_j, 2^{k+2}\sqrt{t})} \int_{B^g(x^t_j, 2^{k+2}\sqrt{t})} |f(x)-f(y) |^2 d\mu_2(x)d\mu_2(y).\\
\end{split}
\end{equation}
\end{lemma}

\begin{proof}
 \begin{equation}
  \begin{split}
   g_0^{j,t}(x)=& f(x)-\frac{1}{\omega(B^g(2\sqrt{t}))}\int_{B^g(x^t_j, 2\sqrt{t}) }f(y)d\mu_2(y)\\
=&\frac{1}{\omega(B^g(2\sqrt{t}))}\int_{B^g(x^t_j, 2\sqrt{t}) }(f(x)-f(y))d\mu_2(y).\\
  \end{split}
\end{equation}
By Cauchy-Schwarz inequality,
\begin{equation}
|g_0^{j,t}(x)|^2\leq \frac{C}{\omega(B^g(4\sqrt{t}))}\int_{B^g(x^t_j, 2\sqrt{t})}|f(x)-f(y)|^2 d\mu_2(y).
\end{equation}
Hence
\begin{equation}  \begin{split}
&\|g_0^{j,t}(x)\|^2_{L^2(C^{j,t}_0, \mu_2)}\\
\leq& \frac{C}{\omega(B^g(2\sqrt{t}))}\int_{B^g(x^t_j, 4\sqrt{t})}\int_{B^g(x^t_j, 4\sqrt{t})}|f(x)-f(y)|^2 
d\mu_2(x)d\mu_2(y).\\
  \end{split}
\end{equation}
By a similar argument, we prove Part B of the lemma as well:
\begin{equation}  \begin{split}
&\|g_k^{j,t}(x)\|^2_{L^2(C^{j,t}_k, \mu_2)}\\
\leq& \frac{C}{\omega(B^g(2^k\sqrt{t}))}
\int_{B^g(x^t_j, 2^{k+2}\sqrt{t})}\int_{B^g(x^t_j, 2^{k+2}\sqrt{t})}|f(x)-f(y)|^2 
d\mu_2(x)d\mu_2(y).\\
  \end{split}
\end{equation}
\end{proof}

\begin{lemma}\label{3.111}Let $\alpha \in (0,2)$. Let $\mathcal{D}(L_{\mu_2})$ be the domain of functions on which $L_{\mu_2}$ is well defined. There exists $C>0$ such that for all $f\in \mathcal{D}(L_{\mu_2})$,
\begin{equation}
\begin{split}
&\int_0^{\infty}t^{-1-\alpha/2}
 \|tL_{\mu_2}(I+tL_{\mu_2})^{-1}f\|^2_{L^2(\mathbb{R}^n, \mu_2)}dt\\
\leq& C\int_{2B}\int_{2B}
\frac{|f(x)-f(y)|^2}{d_g(x,y)^{n+\alpha}}
\omega(x)\omega(y)dxdy.\\
\end{split}
\end{equation}
\end{lemma}
\begin{proof}
Given $t\in (0,+\infty)$,
we will give an upper bound containing the first order difference of $f$.
By Proposition \ref{Proposition1}, volume of geodesic balls on $(M^n, g)$ satisfies the doubling property.
Thus we can find a Vitali's covering of $(M^n, g)=(\mathbb{R}^n, e^{2u}|dx|^2)$: there exists
 a countable family of balls $B^g(x^t_j, \sqrt{t})$, such that they are mutually disjoint, and
$$M^n= \bigcup_{j\in \mathbb{N}} B^g(x^t_j,2 \sqrt{t}). $$
By Lemma \ref{OverlapCovering}, there exists a constant $C$, such that
for all $\theta>1$, and $x\in \mathbb{R}^n$, there are at most $C\theta^{2\kappa}$
indexes $j$ such that $d_g(x,x^t_j)\leq \theta \sqrt{t}$ where $\kappa$ is given by

For a fixed $j$, one has 
$$ tL_{\mu_2}(I+tL_{\mu_2})^{-1}f = tL_{\mu_2} (I+ tL_{\mu_2})^{-1}g^{j,t}
$$
where $g^{j,t}$ is defined to be
$$g^{j,t}(x):=f(x)- m^{j,t}, $$
$$m^{j,t}:= \frac{1}{\omega(B^g(x^t_j,2 \sqrt{t}))} \int_{B^g(x^t_j,2 \sqrt{t})}f(y)d\mu_2(y).$$
Then 
\begin{equation}
\begin{split}
& \|tL_{\mu_2}(I+tL_{\mu_2})^{-1}f\|^2_{L^2(\mathbb{R}^n, \mu_2)}\\
\leq& \sum_{j\in\mathbb{N}}\|tL_{\mu_2}(I+tL_{\mu_2})^{-1}f\|^2_{L^2(B^g(x^t_j,2 \sqrt{t}), \mu_2) } \\
\leq& \sum_{j\in\mathbb{N}}\|tL_{\mu_2}(I+tL_{\mu_2})^{-1}g^{j,t}\|^2_{L^2(B^g(x^t_j,2 \sqrt{t}), \mu_2) }. \\
\end{split}
\end{equation}
To estimate
$$\sum_{j\in\mathbb{N}}\|tL_{\mu_2}(I+tL_{\mu_2})^{-1}g^{j,t}\|^2_{L^2(B^g(x^t_j,2 \sqrt{t}), \mu_2) },
$$
we set
$$C^{j,t}_0 =B^g(x^t_j,4 \sqrt{t}),$$
and 
$$ C^{j,t}_k= B^g(x^t_j,2^{k+2} \sqrt{t})\setminus B^g(x^t_j,2^{k+1} \sqrt{t}),
$$
for $k\geq 1$.
We also define $$g^{j,t}_k(x):=
g^{j.t}(x) \chi_{C^{j,t}_k}(x),$$
where $\chi_{C^{j,t}_k}(x)$ is the characteristic function of a set $C^{j,t}_k$.
Since $g^{j,t}= \sum_{k\leq 0 }g^{j,t}_k$, we have
\begin{equation}
\begin{split}
&\|tL_{\mu_2}(I+tL_{\mu_2})^{-1}g^{j,t}\|_{L^2(B^g(x^t_j,2 \sqrt{t}), \mu_2) }\\
\leq& \sum_{k\geq 0}\|tL_{\mu_2}(I+tL_{\mu_2})^{-1}g^{j,t}_k\|_{L^2(B^g(x^t_j,2 \sqrt{t}), \mu_2) }. \\
\end{split}
\end{equation}

Using Lemma \ref{off} we obtain that
\begin{equation}
\begin{split}
&\|tL_{\mu_2}(I+tL_{\mu_2})^{-1}g^{j,t}\|_{L^2(B^g(x^t_j,2 \sqrt{t}), \mu_2) }\\
\leq&C \left( \|g_0^{j,t}\|_{L^2(C^{j,t}_0, \mu_2) } +\sum_{k\geq 1}  e^{-c 2^k} \|g_k^{j,t}\|_{L^2(C^{j,t}_k, \mu_2) }\right ).\\
\end{split}
\end{equation}
By Cauchy-Schwarz's inequality, we deduce 
\begin{equation}
\begin{split}
&\|tL_{\mu_2}(I+tL_{\mu_2})^{-1}g^{j,t}\|^2_{L^2(B^g(x^t_j,2 \sqrt{t}), \mu_2) }\\
\leq&C' \left( \|g_0^{j,t}\|^2_{L^2(C^{j,t}_0, \mu_2) } +\sum_{k\geq 1}  e^{-c 2^k} \|g_k^{j,t}\|^2_{L^2(C^{j,t}_k, \mu_2) } \right).\\
\end{split}
\end{equation}
Therefore, 
\begin{equation}
\begin{split}
&\int_0^{\infty}t^{-1-\alpha/2}
 \|tL_{\mu_2}(I+tL_{\mu_2})^{-1}f\|^2_{L^2(\mathbb{R}^n, \mu_2)}dt\\
\leq&C'\int_0^{\infty}t^{-1-\alpha/2}
\sum_{j\geq 0}  \|g_0^{j,t}\|^2_{L^2(C^{j,t}_0, \mu_2) }dt +\\
&C'
\int_0^{\infty}t^{-1-\alpha/2}\sum_{k\geq 1}  e^{-c 2^k} \sum_{j\geq 0} 
\|g_k^{j,t}\|^2_{L^2(C^{j,t}_k, \mu_2) } dt.\\
\end{split}
\end{equation}

We now apply Lemma \ref{3.112} to finish the proof. Using part A, and integrating
$t$ over $(0, \infty)$, we get

\begin{equation}
\begin{split}
& \int_{0}^{\infty}t^{-1-\alpha/2}\sum_{j\geq 0}
\|g_0^{j,t} \|^2_{L^2 (C_0^{j,t}, \mu_2)} dt\\
=&
\sum_{j\geq 0} \int_{0}^{\infty} t^{-1-\alpha/2}\|g_0^{j,t}\|^2_{L^2(C_0^{j,t}, \mu_2)}dt\\
\leq & C \sum_{j\geq 0} \int_{0}^\infty 
\frac{t^{-1-\alpha/2}}{\omega(B^g(\sqrt {t}))}\left(  \int_{B^g(x^t_j, 4\sqrt{t})}  \int_{B^g(x^t_j,4\sqrt{t})}
 |f(x)-f(y) |^2 d\mu_2(x)d\mu_2(y) \right)dt \\
\leq &C \sum_{j\geq 0} 
 \int_{\mathbb{R}^n}  \int_{\mathbb{R}^n}
 |f(x)-f(y) |^2 \\
&\cdot \left( \int_{t\geq \max\left\{ \frac{ d_g( x,x^t_j)^2}{16},
 \frac{ d_g(  y, x^t_j)^2}{16},  \right\}
}\frac{ t^{-1-\alpha/2}}{\omega(B^g(2^k\sqrt{t}))}     dt  \right) d\mu_2(x)d\mu_2(y).\\ 
\end{split}
\end{equation}
By Lemma \label{volumegrowth}
$\omega(B^g(x, r))$ satisfies 
$$C_2 r^n \leq \omega(B^g(x,r))\leq C_1 r^n,$$
where $C_1$, $C_2$ are independent of $x$. Therefore, in the above formula, we denote 
volume of a ball with radius $\sqrt{t}$ by $\omega(B^g(\sqrt{t}))$, disregarding where the center is.
Notice that
\begin{equation}
\begin{split}
&\sum_{j\geq 0}  \int_{t\geq \max\left\{ \frac{ d_g(  x, x^t_j)^2}{16},
 \frac{ d_g(  y, x^t_j)^2}{16},  \right\}
}\frac{ t^{-1-\alpha/2}}{\omega(B^g(2^k\sqrt{t}))}     dt = \\ 
& \int_{0}^{\infty}\frac{t^{-1-\alpha/2}}{\omega(B^g(\sqrt {t})) }
\sum_{j\geq 0} \chi_{\left(\max\left\{ \frac{ d_g(  x, x^t_j)^2}{16},
 \frac{ d_g(  y, x^t_j)^2}{16},  \right\}, \infty\right) }(t)dt.\\
\end{split}
\end{equation}
By Lemma \ref{OverlapCovering}, there exists a constant $N\in \mathbb{N}$, such that
for all $t>0$, there are at most $N$ indexes $j$ such that
$d_g( x,x^t_j)<16t$, and $d_g( y,x^t_j)<16t$,and for these indexes $j$,
$ d_g(x, y)<8t$. Thus
\begin{equation}
\begin{split}
&\sum_{j\geq 0} \chi_{\left(\max\left\{ \frac{ d_g( x,x^t_j)^2}{16},
 \frac{ d_g(y, x^t_j)^2}{16},  \right\}, \infty\right) }(t)dt\leq
N \chi_{\left( \frac{ d_g(x,y)^2}{64} , \infty\right) } . \\
\end{split}
\end{equation}
We therefore obtain
\begin{equation}
\begin{split}
& \int_{0}^{\infty}t^{-1-\alpha/2}\sum_{j\geq 0}
\|g_0^{j,t} \|^2_{L^2 (C_0^{j,t}, \mu_2)} dt\\
\leq&
CN 
 \int_{\mathbb{R}^n}  \int_{\mathbb{R}^n}
 |f(x)-f(y) |^2 \\
&\cdot \left( \int_{t\geq  d_g( x,y)^2/64
}\frac{ t^{-1-\alpha/2}}{\omega(B^g(2^k\sqrt{t}))}     dt  \right) d\mu_2(x)d\mu_2(y)\\ 
\leq& CN  \int_{\mathbb{R}^n}  \int_{\mathbb{R}^n}
\frac{ |f(x)-f(y) |^2 }{\omega(B^g_{xy})\cdot d_g( x,y)^{\alpha}}d\mu_2(x)d\mu_2(y).\\ 
\end{split}
\end{equation}
By part B in Lemma \ref{3.112}, we have for all $j\geq 0$ and $k\geq 1$,
\begin{equation}
\begin{split}
& \int_{0}^{\infty}t^{-1-\alpha/2}\sum_{j\geq 0}
\|g_k^{j,t} \|^2_{L^2 (C_0^{j,t}, \mu_2)} dt\\
\leq & C \sum_{j\geq 0}  \int_{0}^{\infty}\frac{ t^{-1-\alpha/2}}{\omega(B^g(2^k\sqrt{t}))} 
 \int_{B^g(x^t_j, 2^{k+2}\sqrt{t})}  \int_{B^g(x^t_j, 2^{k+2}\sqrt{t})}
 |f(x)-f(y) |^2 d\mu_2(x)d\mu_2(y) dt\\ 
\leq& C \sum_{j\geq 0} \int_{\mathbb{R}^n}  \int_{\mathbb{R}^n}
 |f(x)-f(y) |^2 \\
& \cdot  \left( \int_{0}^{\infty}\frac{t^{-1-\alpha/2}}{\omega(B^g(2^k\sqrt {t})) }
\chi_{\left(\max\left\{ \frac{ d_g( x,x^t_j)^2}{4^{k+2}}
 \frac{ d_g( y,x^t_j)^2}{4^{k+2}},  \right\}, \infty\right) }(t)dt\right) d\mu_2(x)d\mu_2(y).\\ 
\end{split}
\end{equation}
With a given $t>0$, $x, y\in \mathbb{R}^n$, by Lemma \ref{3.112}, there exist
at most $C 2^{2k\kappa}$ indexes $j$ such that
$$ d_g(x,x^t_j)\leq 2^{k+2} \sqrt{t}, $$
and 
$$ d_g(y,x^t_j)\leq 2^{k+2} \sqrt{t}, $$
For these $j$'s, $$ d_g(x,y)\leq 2^{k+3} \sqrt{t}.$$
Therefore,
\begin{equation}
\begin{split}
& \int_{0}^{\infty}\frac{t^{-1-\alpha/2}}{\omega(B^g(2^k\sqrt {t})) }
\chi_{\left(\max\left\{ \frac{ d_g( x,x^t_j)^2}{4^{k+2}}
 \frac{ d_g( y,x^t_j)^2}{4^{k+2}},  \right\}, \infty\right) }(t)dt\\ 
\leq &\displaystyle  C  2^{2k\kappa} \int_{t\geq  \frac{ d_g( x,y)^2}{4^{k+3}}} 
\frac{t^{-1-\alpha/2}}{\omega(B^g( 2^k\sqrt {t})) } dt\\
\leq & C \displaystyle \frac{2^{k(2\kappa+\alpha)}}{\omega( B^g_{xy})\cdot d_g(x, y)^{\alpha}}.\\
\end{split}
\end{equation}
We therefore get
\begin{equation}
\begin{split}
&\int_{0}^{\infty}\frac{t^{-1-\alpha/2}}{\omega(B^g(2^k\sqrt {t})) }
\sum_{j\geq 0} \| g_k^{j,t}\|^2_{L^2 (C_0^{j,t}, \mu_2)}dt \\
\leq & C 2^{k(2\kappa+\alpha)} \int_{\mathbb{R}^n}  \int_{\mathbb{R}^n}
\frac{ |f(x)-f(y) |^2 }{\omega(B^g_{xy})\cdot d_g( x,y)^{\alpha}}d\mu_2(x)d\mu_2(y).\\ 
\end{split}
\end{equation}
In conclusion, we have
\begin{equation}\label{last}
\begin{split}
&\int_{0}^{\infty} t^{-1-\alpha/2}
 \|tL_{\mu_2}(I+tL_{\mu_2})^{-1}f\|^2_{L^2(\mathbb{R}^n, \mu_2)}dt\\
\leq &C N \int_{\mathbb{R}^n}  \int_{\mathbb{R}^n}
\frac{ |f(x)-f(y) |^2 }{\omega(B^g_{xy})\cdot d_g( x,y)^{\alpha}}d\mu_2(x)d\mu_2(y)\\ 
&+\sum_{k\geq 1} C 2^{k(2\kappa+\alpha) }e^{-c2^k} \int_{\mathbb{R}^n}  \int_{\mathbb{R}^n}
\frac{ |f(x)-f(y) |^2 }{\omega(B^g_{xy})\cdot d_g( x,y)^{\alpha}}d\mu_2(x)d\mu_2(y)\\ 
\leq &C  \int_{\mathbb{R}^n}  \int_{\mathbb{R}^n}
\frac{ |f(x)-f(y) |^2 }{\omega(B^g_{xy})\cdot d_g( x,y)^{\alpha}}d\mu_2(x)d\mu_2(y)\\
=&C\int_{2B}\int_{2B}
\frac{|f(x)-f(y)|^2}{\omega(B^g_{xy})\cdot d_g(x,y)^{\alpha}}
\omega(x)\omega(y)dxdy.\\ 
\end{split}
\end{equation}
Finally, by the fact that $\omega$ is a strong $A_\infty$ weight, 
$$\omega(B^g_{xy})\approx d_g( x,y)^n.$$
Plugging this into \eqref{last}, we completes the proof of the lemma.
\end{proof}

\hide{\section{Continuous extension}
\begin{theorem}\label{cont extension}
Under the same assumptions as in Theorem \ref{integer}, and in addition let us assume $(M,g)$ is nonpositively curved. Then $(M,g)$
is properly immthe jobersed and the Gauss map extends continuously to the end of $M$. Namely, the Gauss map extends continuously to $\bar{M}$.
\end{theorem}

\section{Example}
For $w=\log\log |x|$, what is the integral of $Q$-curvature, and what is the PDE it satisfies? Is it a normal metric?
}

\end{document}